\documentclass[11pt,oneside]{amsart}
\usepackage{amsmath,ifthen, amsfonts, amssymb,
srcltx,
   amsopn, textcomp}

\usepackage{hyperref}
\usepackage{graphicx}
\usepackage{xypic}

\usepackage[small]{caption}
\usepackage{marginnote}

\newtheorem{thm}{Theorem}[section]
\newtheorem{lem}[thm]{Lemma}

\newtheorem{cor}[thm]{Corollary}

\newtheorem*{claim}{Claim}

\newtheorem*{corA}{Corollary~\ref{cor:geometricFlat}}
\newtheorem*{thmA}{Theorem~\ref{thm:main}}

\theoremstyle{definition}
\newtheorem{defn}[thm]{Definition}

\newtheorem{exmp}[thm]{Example}

\newtheorem{prob}[thm]{Problem}

\DeclareMathOperator{\stabilizer}{Stabilizer}

\newcommand{\dist}{\textup{\textsf{d}}}

\newcommand{\field}[1]{\mathbb{#1}}
\newcommand{\integers}{\ensuremath{\field{Z}}}

\newcommand{\naturals}{\ensuremath{\field{N}}}
\newcommand{\reals}{\ensuremath{\field{R}}}

\newcommand{\ola}[1]{\overleftarrow{ #1 }}
\newcommand{\ora}[1]{\overrightarrow{ #1 }}

 \label{mainB}

\DeclareMathOperator{\stab}{Stab}

\newcommand\restr[2]{{
  \left.\kern-\nulldelimiterspace 
  #1 
  \vphantom{\big|} 
  \right|_{#2} 
  }}

\begin{document}
\title{A Generalized Axis Theorem for Cube Complexes}
\author{Daniel J. Woodhouse}
\email{daniel.woodhouse@mail.mcgill.ca}

\begin{abstract}
 We consider a finitely generated virtually abelian group $G$ acting properly and without inversions on a CAT(0) cube complex $X$.
 We prove that $G$ stabilizes a finite dimensional CAT(0) subcomplex $Y \subseteq X$ that is isometrically embedded in the combinatorial metric.
 Moreover, we show that $Y$ is a product of finitely many quasilines.
 The result represents a higher dimensional generalization of Haglund's axis theorem.
\end{abstract}

\maketitle

\section{Introduction}

A \emph{CAT(0) cube complex} $X$ is a cell complex that satisfies two properties: it is a geodesic metric space satisfying the CAT(0) comparison triangle condition, and each $n$-cell is isometric to $[0,1]^n$.
We will call this metric the \emph{CAT(0) metric} $\dist_X$ and refer to~\cite{BridsonHaefliger} for a comprehensive account.
A \emph{hyperplane} $\Lambda \subseteq X$ is the subset of points equidistant between two adjacent vertices.
Despite the brevity of this definition, hyperplanes are better understood via their combinatorial definition, and the reader is urged to consult the literature; see~\cite{Sageev97}~\cite{HaglundSemiSimple}~\cite{WiseCBMS2012} for the required background.
There also exists an alternative metric on the $0$-cubes of $X$, that we will refer to as the \emph{combinatorial metric} $\dist^c_X$, sometimes referred to as the \emph{$\ell^1$-metric}.
The combinatorial distance between two $0$-cubes is the length of the shortest combinatorial path in $X$ joining the $0$-cubes.
Equivalently, the combinatorial distance between two $0$-cubes is the number of hyperplanes in $X$ separating them.
We will always assume that a group $G$ acting on a CAT(0) cube complex preserves its cell structure and maps cubes isometrically to cubes.
A group $G$ acts without \emph{inversions} if the stabilizer of a hyperplane also stabilizes each  complementary component.
The requirement that the action be without inversions is not a serious restriction as $G$ acts without inversions on the cubical subdivision.

A connected CAT(0) cube complex $X$ is a \emph{quasiline} if it is quasiisometric to $\reals$. 
The \emph{rank} of a virtually abelian group commensurable to $\mathbb{Z}^n$ is $n$.
The goal of this paper will be the following theorem:

\begin{thmA}
Let $G$ be virtually $\mathbb{Z}^n$.
Suppose $G$ acts properly and without inversions on a CAT(0) cube complex $X$.
Then $G$ stabilizes a finite dimensional subcomplex $Y \subseteq X$ that is isometrically embedded in the combinatorial metric, and $Y \cong \prod_{i=1}^m C_i$, where each $C_i$ is a cubical quasiline and $m \geq n$.
Moreover, $\stab_G(\Lambda)$ is a codimension-1 subgroup for each hyperplane $\Lambda$ in $Y$.
\end{thmA}

\noindent
 Note that $Y$ will not in general be a convex subcomplex.

 \begin{cor}
 Let $A$ be a finitely generated virtually abelian group acting properly on a CAT(0) cube complex $X$.
 Then $A$ acts metrically properly on $X$.
 \end{cor}
 
 \begin{cor}
 	Let $G$ be a finitely generated group acting properly on a CAT(0) cube complex $X$.
 	Then virtually $\mathbb{Z}^n$ subgroups are undistorted in $G$. 
 \end{cor}

Let $g$ be an isometry of $X$, and let $x \in X$.
The \emph{displacement of $g$ at $x$}, denoted $\tau_x(g)$, is the distance $\dist_X(x, gx)$.
The \emph{translation length} of $g$, denoted $\tau(g)$, is $\inf\{ \tau_x(g) \mid x \in X \}$.
Similarly, if $x$ is a $0$-cube of $X$, we can define the \emph{combinatorial displacement of $g$ at $x$}, denoted $\tau^c_x(g)$, as $\dist^c_X(x, gx)$ and the \emph{combinatorial translation length}, denoted $\tau^c(g)$, is $\inf\{ \tau^c_x(g) \mid x \in X \}$.
Note that $\tau$, and $\tau^c$ are conjugacy invariant.
An isometry $g$ of a CAT(0) space is \emph{semisimple} if $\tau_x(g) = \tau(g)$ for some $x \in X$, and $G$ acts \emph{semisimply} on a CAT(0) space $X$ if each $g \in G$ is semisimple.

If a virtually $\mathbb{Z}^n$ group $G$ acts metrically properly by semisimple isometries on a CAT(0) space $X$, then the Flat Torus Theorem~\cite{BridsonHaefliger} provides a $G$-invariant, convex, flat $\mathbb{E}^n \subseteq X$.
A group acting on a CAT(0) cube complex does not, in general, have to do so semisimply.
See~\cite{AlgomKfirWajnrybWitowicz13} for examples of non-semisimple isometries in Thompson's group $F$ acting on an infinite dimensional CAT(0) cube complex.
Alternatively, in~\cite{Gersten94} a free-by-cyclic group $G$ is shown not to permit a semisimple action on a CAT(0) space.
Yet in~\cite{WiseGerstenRevisited} it is shown that $G$ does act freely on a CAT(0) cube complex.
Thus Theorem~\ref{thm:main} can be applied to such actions, whereas the classical Flat Torus Theorem cannot.

A virtually abelian subgroup is \emph{highest} if it is not virtually contained in a higher rank abelian subgroup.
 If $G$ is a highest virtually abelian subgroup of a group acting properly and cocompactly on a CAT(0) cube complex $X$, then $G$ cocompactly stabilizes a convex subcomplex $Y$ which is a product of quasilines, as above~\cite{WiseWoodhouse15}.
However, this theorem fails without the highest hypothesis.
Moreover, most actions do not arise in the above fashion.

 Despite the fact that the Flat Torus Theorem will not hold under the hypotheses of Theorem~\ref{thm:main}, we can deduce the following:

\begin{corA}
Let $G$ be virtually $\integers^n$.
Suppose $G$ acts properly and without inversions on a CAT(0) cube complex $X$.
Then $G$ cocompactly stabilizes a subspace  $F \subseteq X$ homeomorphic to $\mathbb{R}^n$ such that for each hyperplane $\Lambda \subseteq X$, the intersection $\Lambda \cap F$ is either empty or homeomorphic to $\mathbb{R}^{n-1}$.
\end{corA}


The initial motivation for Theorem~\ref{thm:main} and Corollary~\ref{cor:geometricFlat} was to resolve the following question posed by Wise.
Although we have not found a combinatorial flat, Corollary~\ref{cor:geometricFlat} is perhaps better suited to applications (see~\cite{Woodhouse15b}).

  \begin{prob} \label{Problem:Wise}
    Let $\mathbb{Z}^2$ act freely on a CAT(0) cube complex $Y$.
    Does there exists a $\integers^2$-equivariant map $F \rightarrow Y$ where $F$ is a square $2$-complex homeomorphic to $\reals^2$, and such that no two hyperplanes of $F$ map to the same hyperplane in $Y$?
  \end{prob}

   A \emph{combinatorial geodesic axis for $g$} is a $g$-invariant, isometrically embedded, subcomplex $\gamma \subseteq X$ with $\gamma \cong \mathbb{R}$.
Note that $\gamma$ realizes the minimal combinatorial translation length of $g$.
Theorem~\ref{thm:main} is a high dimensional generalization of Haglund's combinatorial geodesic axis theorem.
Haglund's proof involved an argument by contradiction, exploiting the geometry of hyperplanes.
 We reprove the result in Section~\ref{sec:HaglundRevisited} by using the dual cube complex construction of Sageev.
The results are further support for Haglund's slogan ``in CAT(0) cube complexes the combinatorial geometry is as nice as the CAT(0) geometry''.

 The following is an application of Theorem~\ref{thm:main}, and the argument is inspired by the solvable subgroup theorem~\cite{BridsonHaefliger}.
 Note that since we do not require that the action of $G$ on a CAT(0) cube complex be semisimple the following is not covered by the solvable subgroup theorem.

\begin{cor}
Let $H$ be virtually $\mathbb{Z}^n$, and let $\phi : H \rightarrow H$ be an injection with $\phi \neq \phi^i$ for all $i >1$.
Then $G = \langle H ,t \mid t^{-1} h t = \phi(h): \; h\in H \rangle$ cannot act properly on a CAT(0) cube complex.
\end{cor}

\begin{proof}
Suppose that $G$ acts properly on a CAT(0) cube complex $X$.
After subdividing $X$ we can assume that $G$ acts without inversions.
As $H$ is finitely generated, there exists an $a$ in the finite generating set such that $\phi^i(a) \neq a$ for all $i \in \mathbb{N}$, otherwise $\phi^i = \phi$ for some $i$, contradicting our hypothesis.
Thus, $| \{\phi^i(a)\}| = \infty$.
By Theorem~\ref{thm:main} there is an $H$-equivariant isometrically embedded subcomplex $Y \subseteq X$ such that $Y \cong \prod_{i=1}^m C_i$ where each $C_i$ is a cubical quasiline.

As $Y$ is isometrically embedded in $X$ in the combinatorial metric, the combinatorial translation length $\tau^c(\phi^i(a))$ is the same in $Y$ as it is in $X$.
The set $\{\tau^c(\phi^i(a))\}_{i \in \mathbb{N}}$ must be unbounded since the action of $H$ on $Y$ is proper and $Y$ is locally finite.
However, since $\tau^c$ is conjugacy invariant in $G$, we conclude that $\tau^c(\phi^i(a)) = \tau^c(\phi^j(a))$ for all $i,j \in \mathbb{N}$.
Thus, we arrive at the contradiction that $\{\tau^c(\phi^i(a)) \}_{i\in\mathbb{N}}$ is both bounded and unbounded.
\end{proof}

However, we have the following example of a solvable group which does act freely on a CAT(0) cube complex.

\begin{exmp} \label{ex:solvableCubulated}
Let $H = \langle a_1, a_2, \ldots \mid [a_i , a_j] : i\neq j\rangle$.
Note that $H$ is the fundamental group of the nonpositively curved cube complex $Y$ obtained from a $0$-cube $v$, and $1$-cubes $e_1, e_2, e_3 \ldots$ with $n$-cubes inserted for every cardinality $n$ collection of $1$-cubes to create an $n$-torus.
One should think of $Y$ as an infinite cubical torus.
The oriented loop $e_i$ represents the element $a_i$.

Let $\phi : H \rightarrow H$ be the monomorphism such that $\phi(a_i) = a_{i+1}$.
Let $G = H \ast_\phi = \langle t, a_1, a_2, \ldots \mid [a_i, a_j] : i \neq j\;, t^{-1} a_i t = a_{i+1} \rangle$ be the associated ascending HNN extension.
Note that $G$ is generated by $a_1$ and $t$.
There is a graph of spaces $X$ obtained by letting $Y$ be the vertex space and $Y \times [0, 1]$ be the edge space and identifying $(v,1)$ and $(v,0)$ with $v$, and the $1$-cube $e_i \times \{1\}$ with $e_i$ and $e_i \times \{ 0 \}$ with $e_{i+1}$.
Note that $X$ is nonpositively curved, and therefore $G = \pi_1X$ acts freely on the CAT(0) cube complex $\widetilde X$, the universal cover of $X$.
\end{exmp}


{\bf Acknowledgements: } I would like to thank Daniel T. Wise, Mark F Hagen, Jack Button, Piotr Przytycki, and Dan Guralnik.

\section{Dual Cube Complexes}

 Let $S$ be a set.
 A \emph{wall} $\Lambda = \{ \ola{\Lambda}, \ora{\Lambda} \}$ in $S$ is a partition of $S$ into two disjoint, nonempty subsets.
 The subsets $\ola{\Lambda}, \ora{\Lambda}$ are the \emph{halfspaces} of $\Lambda$.
 A wall $\Lambda$ \emph{separates} $x,y\in S$ if they belong to distinct halfspaces of $\Lambda$.
 Let $K \subseteq S$.
 A wall $\Lambda$ \emph{intersects} $K$ if $K$ nontrivially intersects both $\ola{\Lambda}$ and $\ora{\Lambda}$.
 Let $\mathcal{W}$ be a set of walls in $S$, then $(S,\mathcal{W})$ is a wallspace if for all $x,y \in S$, the number of walls separating $x$ and $y$ is finite.
 If $\Lambda$ intersects $K$, then the \emph{restriction of $\Lambda$ to $K$}, is the wall in $K$ determined by $\restr{\Lambda}{K} = \{ \ola{\Lambda}\cap K, \ora{\Lambda} \cap K\}$.

 In this paper duplicate walls are not permitted in $\mathcal{W}$.
 Let $\mathcal{H}$ be the set of halfspaces of corresponding to $\mathcal{W}$.

  \begin{exmp} \label{exmp:CATCubeComplexes}
  Let $X$ be a CAT(0) cube complex, and let $\Lambda \subseteq X$ be a hyperplane in $X$.
  The complement $X - \Lambda$ has two components, therefore defining a wall in $X$ such that $\ola{\Lambda}$ is an open halfspace not containing $\Lambda$, and $\ora{\Lambda}$ is a closed halfspace containing $\Lambda$. Note that $\ola{\Lambda} \sqcup \ora{\Lambda} = X$.
  Let $L(\Lambda)$ and $R(\Lambda)$ denote the maximal subcomplexes contained in $\ola{\Lambda}$ and $\ora{\Lambda}$ respectively.
  Note that $L(\Lambda)$ and $R(\Lambda)$ are convex subcomplexes.
  Let $\mathcal{W}$ be the set of walls determined by the hyperplanes in $X$.
  Then $(X, \mathcal{W})$ is the wallspace associated to $X$.
  Note that we are using $\Lambda$ to denote both the hyperplane and the wall corresponding to the hyperplane.
  \end{exmp}

  A function $c:\mathcal{W} \rightarrow \mathcal{H}$ is a \emph{$0$-cube} if $c[\Lambda] \in \{\ola{\Lambda}, \ora{\Lambda}\}$ and the following two conditions are satisfied:
 \begin{enumerate}
  \item \label{ax:intersection} For all $\Lambda_1, \Lambda_2 \in \mathcal{W}$ the intersection $c[\Lambda_1] \cap c[\Lambda_2]$ is nonempty.
  \item \label{ax:finiteDisparity} For all $x \in S$, the set $\{ \Lambda \in \mathcal{W} \mid x \notin c[\Lambda] \}$ is finite.
 \end{enumerate}

 The \emph{dual cube complex} $C(S, \mathcal{W})$ is the connected CAT(0) cube complex obtained by letting the union of all $0$-cubes be the $0$-skeleton.
 Two $0$-cubes $c_1 \neq c_2$ are endpoints of a $1$-cube if $c_1[\Lambda] = c_2[\Lambda]$ for all but precisely one $\Lambda \in \mathcal{W}$.
 An $n$-cube is then inserted wherever there is the $1$-skeleton of an $n$-cube.
 The hyperplanes in $C(S, \mathcal{W})$ are identified naturally with the walls in $\mathcal{W}$.
 A proof of the fact that $C(S, \mathcal{W})$ is in fact a CAT(0) cube complex can be found in \cite{Sageev95}.

 A point $x \in S$ determines a $0$-cube $c_x$ defined such that $x \in c_x[\Lambda]$ for all $\Lambda \in \mathcal{W}$.
 Condition~\eqref{ax:intersection} holds immediately since $x \in c_x[\Lambda]$ for all $\Lambda \in \mathcal{W}$.
 Condition~\eqref{ax:finiteDisparity} holds for $c_x$, since if $y \in S$  a wall $\Lambda$ does not separate $x$ and $y$, we can deduce that $y \in c_x[\Lambda]$, hence all but finitely many $\Lambda$ satisfy $y \in c_x[\Lambda]$.
 Such $0$-cubes are called the \emph{canonical $0$-cubes}.


 \begin{lem} \label{lem:hemi0}
 Let $X$ be a CAT(0) cube complex.
 Let $\mathcal{W}$ be a set of walls obtained from the hyperplanes in $X$.
 Let $Z$ be a connected subcomplex of $X$, and let $\mathcal{W}_{Z} \subseteq \mathcal{W}$ be the subset of walls intersecting $Z$.
 Let $\mathcal{V}$ be walls in $\mathcal{W}_{Z}$ restricted to $Z$.
 Then $(Z, \mathcal{V})$ is a wallspace and $C(Z, \mathcal{V})$ embeds in $C(X, \mathcal{W})$ isometrically in the combinatorial metric.
 \end{lem}

 \begin{proof}
  We first claim that the map $\mathcal{W}_Z \rightarrow \mathcal{V}$ is an injection.
  Suppose that $\Lambda_1, \Lambda_2 \in \mathcal{W}_Z$ are distinct walls.
  As $\Lambda_1, \Lambda_2$ intersects $Z$, and since $Z$ is connected, there are $1$-cubes $e_1, e_2$ in $Z$ that are dual to the hyperplanes corresponding to $\Lambda_1, \Lambda_2$.
  Therefore, both $0$-cubes in $e_1$ belong in a single halfspace of $\restr{\Lambda_2}{Z}$, so $\restr{\Lambda_1}{Z} \neq \restr{\Lambda_2}{Z}$.

  We construct a map $\phi: C(Z, \mathcal{V}) \rightarrow C(X, \mathcal{W})$ on the $0$-skeleton first.
  Let $c$ be a $0$-cube in $C(Z, \mathcal{V})$.
  We let $\phi(c) \in C(X, \mathcal{W})$ be the uniquely defined $0$-cube such that $\phi(c)[\Lambda] \supseteq c[\restr{\Lambda}{Z}]$ for $\restr{\Lambda}{Z} \in \mathcal{V}$, and $\phi(c)[\Lambda] \supseteq Z$ for $\Lambda \in \mathcal{W} - \mathcal{W}_{Z}$.
  To verify that $\phi(c)$ is a $0$-cube, first observe that $\phi(c)[\Lambda_1] \cap \phi(c)[\Lambda_2]$ is nonempty since $\restr{\Lambda_1}{Z} \cap \restr{\Lambda_2}{Z} \subseteq X$.
  Secondly, if $x \in X$ we need to show that $x \in \phi(c)[\Lambda]$ for all but finitely many $\Lambda \in \mathcal{W}$.
  Choose $z \in Z$, then $z \in c[\restr{\Lambda}{Z}]$ for all $\restr{\Lambda}{Z} \in \mathcal{V} - \{\restr{\Lambda_1}{Z}, \ldots, \restr{\Lambda_k}{Z} \}$, hence $z \in \phi(c)[\Lambda]$ for all $\Lambda \in \mathcal{W}_{Z} - \{\Lambda_1, \ldots, \Lambda_k\}$.
  Let $\{\Lambda_{k+1}, \ldots, \Lambda_{k + \ell} \}$ be the set of walls in $\mathcal{W}$ separating $x$ and $z$.
  Then $x \in \phi(c)[\Lambda]$ for all $\Lambda \in \mathcal{W} - \{ \Lambda_1, \ldots \Lambda_{k+\ell}\}$.

  The $0$-cubes are embedded since if $c_1 \neq c_2$, there exists $\restr{\Lambda}{Z} \in \mathcal{V}$ such that $c_1[\restr{\Lambda}{Z}] \neq c_2[\restr{\Lambda}{Z}]$, hence $\phi(c_1)[\Lambda] \neq \phi(c_2)[\Lambda]$.
  If $c_1, c_2$ are adjacent $0$-cubes in $C(Z, \mathcal{V})$, then $c_1[\restr{\Lambda}{Z}] = c_2[\restr{\Lambda}{Z}]$ for all $\restr{\Lambda}{Z} \in \mathcal{V}$, with the exception of precisely one wall $\restr{\hat{\Lambda}}{Z}$.
  Therefore, we can deduce that $\phi(c_1)[\Lambda] = \phi(c_2)[\Lambda]$ for all walls in $\mathcal{W}$, with the precise exception of $\hat{\Lambda}$.
  Therefore, the $1$-skeleton of $C(Z, \mathcal{V})$ embeds in $C(X,\mathcal{W})$, which is sufficient for $\phi$ to extend to an embedding of the entire cube complex.

  Consider $C(Z, \mathcal{V})$ as a subcomplex of $C(X, \mathcal{W})$.
  The set of hyperplanes in $C(Z, \mathcal{V})$ embeds into the set of hyperplanes in $C(X, \mathcal{W})$.
  To see that $C(Z, \mathcal{V})$ is an isometrically embedded subcomplex, let $z_1, z_2$ be $0$-cubes in $Z$ and $\gamma$ be a geodesic combinatorial path in $C(Z, \mathcal{V})$ joining them.
  Each hyperplane dual to $\gamma$ in $C(Z,\mathcal{V})$ intersects $\gamma$ precisely once, and since the hyperplanes in $C(Z,  \mathcal{V})$ inject to hyperplanes in $C(X, \mathcal{W})$, it is geodesic there as well.
 \end{proof}

 Given a wall $\Lambda$ associated to a hyperplane in $X$ we let $N(\Lambda)$ denote the \emph{carrier} of $\Lambda$, by which we mean the union of all cubes intersected by $\Lambda$.
 
 The following Lemma decribes what is called the \emph{restriction quotient} in~\cite{CapraceSageev2011}.

 \begin{lem} \label{lem:hemi1}
 Let $S$ be a set and let $\mathcal{W}$ be a set of walls of $S$.
 Let $G$ be a group acting on $(S, \mathcal{W})$.
 Let $\mathcal{V} \subseteq \mathcal{W}$ be a $G$-invariant subset.
 Then there is a $G$-equivariant function $\phi:C(S, \mathcal{W})^0 \rightarrow C(S, \mathcal{V})^0$.
 Moreover, $\phi^{-1}(z)$ is nonempty for all $0$-cubes $z$ in $C(S, \mathcal{V})$.
 \end{lem}

 \begin{proof}
 Let $c$ be a $0$-cube in $C(S, \mathcal{W})$.
 Let $\phi(c)[\Lambda] = c[\Lambda]$ for $\Lambda \in \mathcal{V}$.
 It is immediate that $\phi$ is $G$-equivariant.

 To verify $\phi(c)[\Lambda]$ is a $0$-cube in $C(S, \mathcal{V})$ first note that $\phi(c_1)[\Lambda_1] \cap \phi(c_2)[\Lambda_2] \neq \emptyset$  for all $\Lambda_1, \Lambda_2 \in \mathcal{V}$, since $c_1[\Lambda_1] \cap c_2[\Lambda_2] \neq \emptyset$ for all $\Lambda_1, \Lambda_2 \in \mathcal{W}$.
 Secondly, for all $x \in S$ observe that $x \in \phi(c)[\Lambda]$ for all but finitely many $\Lambda \in \mathcal{V}$.
 Indeed, this is true for all but finitely many $\Lambda \in \mathcal{W}$.

 To see that $\phi^{-1}(z)$ is non-empty for all $0$-cubes $z$ in $C(S, \mathcal{V})$ we determine a $0$-cube $x$ in $C(S, \mathcal{W})$ such that $\phi(x) = z$.
 Fix $s \in S$.
 Let $x[\Lambda] = z[\Lambda]$ for $\Lambda \in \mathcal{V}$.
 Suppose that $\Lambda \in \mathcal{W} - \mathcal{V}$.
 If $\ora{\Lambda} \supseteq z[\Lambda']$ for some $\Lambda' \in \mathcal{V}$ let $x[\Lambda] = \ora{\Lambda}$.
 Similarly if $\ola{\Lambda} \supseteq z[\Lambda']$.
 Otherwise, if $\Lambda$ intersects $z[\Lambda']$ for all $\Lambda' \in \mathcal{V}$ then let $s \in x[\Lambda]$.

 To verify that $x$ is a $0$-cube, consider the following cases to show $x[\Lambda_1] \cap x[\Lambda_2] \neq \emptyset$ for $\Lambda_1, \Lambda_2 \in \mathcal{W}$.
 If $\Lambda_1, \Lambda_2 \in \mathcal{V}$ then $x[\Lambda_1] \cap x[\Lambda_2] = z[\Lambda_1] \cap z[\Lambda_2] \neq \emptyset$.
 Suppose that $\Lambda_1 \in \mathcal{W} - \mathcal{V}$ and $x[\Lambda_1] \subseteq z[\Lambda_1']$ for some $\Lambda_1' \in \mathcal{V}$.
 If $\Lambda_2 \in \mathcal{V}$, then $x[\Lambda_1] \cap x[\Lambda_2] \supseteq z[\Lambda_1'] \cap z[\Lambda_2] \neq \emptyset$.
 If $\Lambda_2 \in \mathcal{W} - \mathcal{V}$ and $x[\Lambda_2] \subseteq z[\Lambda_2']$ for some $\Lambda_2' \in \mathcal{V}$ then $x[\Lambda_1] \cap x[\Lambda_2] \subseteq z[\Lambda_1'] \cap z[\Lambda_2'] \neq \emptyset$.
 If $\Lambda_2$ intersects $z[\Lambda]$ for all $\Lambda \in \mathcal{V}$, then $x[\Lambda_1] \cap x[\Lambda_2] \supseteq z[\Lambda_1'] \cap x[\Lambda_2] \neq \emptyset$.
 Finally if both $s \in x[\Lambda_1]$ and $ x[\Lambda_2]$, then their intersection will contain at least $s$.

 Finally, we verify that for $s' \in S$ there are only finitely many $\Lambda \in \mathcal{W}$ such that $s' \notin x[\Lambda]$.
 Suppose, by way of contradiction, that there is an infinite subset of walls $\{ \Lambda_1, \Lambda_2, \ldots \} \subseteq \mathcal{W}$ such that $s' \notin x[\Lambda_i]$ for all $i \in \naturals$.
 We can assume, by excluding at most finitely many walls, that each $\Lambda_i \in \mathcal{W} - \mathcal{V}$.
 Similarly, by excluding finitely many walls, we can assume that $\Lambda_i$ does not separate $s$ and $s'$.
 Therefore, $s \notin x[\Lambda_i]$ for $i \in \naturals$.
 Therefore, by construction of $x$, there exist $\Lambda_i' \in \mathcal{V}$ such that $z[\Lambda_i'] \subseteq x[\Lambda_i]$, which implies that $s' \notin z[\Lambda_i']$.
 There are infinitely many distinct $\Lambda_i'$, as otherwise there is a $\Lambda' \in \mathcal{V}$ such that $z[\Lambda']\subseteq x[\Lambda_i]$ for infinitely many $i$, which would imply that infinitely many $\Lambda_i$ separate $s'$ from an element in the complement of $z[\Lambda']$.
 Therefore, infinitely many distinct walls $\Lambda_i' \in \mathcal{V}$ have $s' \notin z[\Lambda_i']$, contradicting that $z$ is a $0$-cube in $C(S, \mathcal{V})$.
 \end{proof}

\section{Minimal $\mathbb{Z}^n$-invariant convex subcomplexes}

 The following is Theorem 2 from~\cite{Gerasimov97}. As this paper is written in Russian, we give a proof in Appendix~\ref{AppendixA} based on the work in~\cite{NibloRoller98} as well as stating the definition of codimension-1.

 \begin{thm}[Gerasimov~\cite{Gerasimov97}]\label{thm:fixedCube}
  Let $G$ be a finitely generated group that acts on a CAT(0) cube complex $X$ without a fixed point or inversions.
  Then there is a hyperplane in $X$ that is stabilized by a  codimension-1 subgroup of $G$.
  \end{thm}

\noindent The goal of this section is to prove the following:

\begin{lem} \label{lem:minSubcomplex}
 Let $G$ be a finitely generated group acting without fixed point or inversions on a CAT(0) cube complex $X$.
 There exists a minimal, $G$-invariant, convex subcomplex  $X_o \subseteq X$ such that $X_o$ contains only finitely many hyperplane orbits, and every $X_o$ hyperplane stabilizer is a codimension-1 subgroup of $G$.
\end{lem}

\begin{proof}
 Since $G$ is finitely generated, by taking the convex hull of a $G$-orbit we obtain a $G$-invariant convex subcomplex $X_o \subseteq X$ containing finitely many $G$-orbits of hyperplanes.
 Assume that $X_o$ is a minimal such subcomplex in terms of the number of hyperplane orbits.

 Let $(X_o, \mathcal{W})$ be the wallspace obtained from the hyperplanes in $X_o$.
 Suppose that $\stab_G(\Lambda)$ is not a codimension-1 subgroup of $G$ for some $\Lambda \in \mathcal{W}$.
 Let $G \Lambda \subseteq \mathcal{W}$ be the $G$-orbit of $\Lambda$.
 By Lemma~\ref{lem:hemi1} there is an $G$-invariant map $\phi: X_o^{0} \rightarrow C(X_o, G\Lambda)^{0}$.
 Since $\stab_G(\Lambda)$ is not commensurable to a codimension-1 subgroup, Theorem~\ref{thm:fixedCube} implies that there is a fixed $0$-cube $x$ in $C(X_o, G\Lambda)$.
  Lemma~\ref{lem:hemi1} then implies that $\phi^{-1}(x)$ is non-empty.
 Assuming that $\phi^{-1}(x) \subseteq \ola{\Lambda}$, then the intersection $\bigcap_{g\in G}g L(\Lambda)$ contains a proper, convex, $G$-invariant subcomplex of $X_o$, with one less hyperplane orbit.
 This contradicts the minimality of $X_o$.
\end{proof}

 The following Corollary follows since all codimension-1 subgroups of a rank $n$ virtually abelian group are of rank $(n-1)$.

 \begin{cor} \label{cor:abelianMinSubcomplex}
 Let $G$ be a rank $n$, virtually abelian group acting without fixed point or inversions on a CAT(0) cube complex $X$.
 Then there exists a minimal, $G$-invariant, convex subcomplex  $X_o \subseteq X$ such that $X_o$ contains only finitely many hyperplane orbits, and every hyperplane stabilizer is a rank $(n-1)$ subgroup of $G$.
 \end{cor}

\section{Proof of Main Theorem}

\begin{defn}
Regard $\reals$ as a CAT(0) cube complex whose $0$-skeleton is $\integers$.
Let $g$ be an isometry of $X$.
A \emph{geodesic combinatorial axis} for $g$ is a $g$-invariant subcomplex homeomorphic to $\reals$ that embeds isometrically in $X$.
\end{defn}
\begin{defn}
Let $(M, d)$ be a metric space.
The subspaces $N_1, N_2 \subseteq M$ are \emph{coarsely equivalent} if each lies in an $r$-neighbourhood of the other for some $r>0$.
\end{defn}

\begin{thm} \label{thm:main}
Let $G$ be virtually $\mathbb{Z}^n$.
Suppose $G$ acts properly and without inversions on a CAT(0) cube complex $X$.
Then $G$ stabilizes a finite dimensional subcomplex $Y \subseteq X$ that is isometrically embedded in the combinatorial metric, and $Y \cong \prod_{i=1}^m C_i$, where each $C_i$ is a cubical quasiline and $m \geq n$.
Moreover, $\stab_G(\Lambda)$ is a codimension-1 subgroup for each hyperplane $\Lambda$ in $Y$.
\end{thm}

\begin{proof}
By Corollary~\ref{cor:abelianMinSubcomplex} there is a minimal, non-empty, convex subcomplex $X_o \subseteq X$ stabilized by $G$, containing finitely many hyperplane orbits, and $\stabilizer_G(\Lambda)$ is a rank $(n-1)$ subgroup of $G$, for each hyperplane $\Lambda \subseteq X_o$.

Let $S = \{ g_1 ,\ldots, g_r \}$ be a generating set for $G$.
Let $x \in X_o$ be a $0$-cube.
Let $\Upsilon$ be the Cayley graph of $G$ with respect to $S$.
Let $\phi : \Upsilon \rightarrow X_o$ be a $G$-equivariant map that sends vertices to vertices, and edges to combinatorial paths or vertices in $X_o$.
%
Let $Q = \phi(\Upsilon)$.
As  $G$ acts properly on $X$, and cocompactly on $\Upsilon$, the graph $Q $ is quasiisometric to $G$.
Let $\mathcal{W}_Q$ be the set of hyperplanes intersecting $Q$, and let $(Q, \mathcal{W}_Q)$ be the associated wallspace.
By Lemma~\ref{lem:hemi0} we know that $C(Q, \mathcal{W}_Q)$ is an isometrically embedded subcomplex of $X_o$. 
Fix a proper action of $G$ on $\mathbb{R}^n$, and let $q: Q \rightarrow \mathbb{R}^n$ be a $G$-equivariant quasiisometry.
Note that $\stabilizer_G(\Lambda)$ is a quasiisometrically embedded subgroup of $G$, for all $\Lambda \in \mathcal{W}_Q$.
  Thus $q(\Lambda \cap Q)$ is coarsely equivalent to a codimension-1 affine subspace $H \subseteq \mathbb{R}^n$.
Moreover, $q(\ola{\Lambda} \cap Q)$ and $q(\ora{\Lambda} \cap Q)$ are coarsely equivalent to the halfspaces of $H$.


Let $n>0$.
Since there are finitely many orbits of hyperplanes in $X_o$, there are only finitely many commensurability classes of stabilizers.
Therefore, we may partition $\mathcal{W}_Q$ as the disjoint union $ \bigsqcup_{i=1}^m \mathcal{W}_i$ where each $\mathcal{W}_i$ contains all walls with commensurable stabilizers.
For each $\Lambda_i \in \mathcal{W}_i$ let $q(\Lambda_i \cap Q)$ be coarsely equivalent to a codimension-1 affine subspace $H_i \subseteq \mathbb{R}^n$, stabilized by $\stabilizer_G(\Lambda_i)$.
If $i\neq j$ then $H_i$ and $H_j$ are nonparallel affine subspaces, and therefore $\Lambda_i$ and $\Lambda_j$ will intersect in $Q$.
Therefore, every wall in $\mathcal{W}_i$ intersects every wall in $\mathcal{W}_j$ if $i \neq j$, and thus $C(Q, \mathcal{W}_Q) \cong \prod_{i=1}^m C(Q, \mathcal{W}_i)$.

Finally, we show that $C(Q, \mathcal{W}_i)$ is a quasiline for each $1 \leq i \leq m$.
As $G$ permutes the factors in $\prod_{i=1}^m C(Q, \mathcal{W}_i)$, there is a finite index subgroup $G' \leqslant G$ that preserves each factor.
For each $i$, the stabilizers $\stab_G(\Lambda)$ are commensurable for all $\Lambda \in \mathcal{W}_i$.
 Therefore, there is a cyclic subgroup $Z_i$ that is not virtually contained in any $\stab_G(\Lambda)$ and thus acts freely on $C(Q, \mathcal{W}_i)$.
As the stabilizers of $\Lambda \in \mathcal{W}_i$ are commensurable, all $q(\Lambda \cap Q)$ will be quasi-equivalent to parallel codimension-1 affine subspaces of $\reals^n$, which implies that only finitely many $Z_i$-translates of $\Lambda$ can pairwise intersect.
As there are finitely many $Z_i$-orbits of $\Lambda$ in $\mathcal{W}_i$, there is an upper bound on the number of pairwise intersecting hyperplanes in $\mathcal{W}_i$.
Thus, there are finitely many $Z_i$-orbits of maximal cubes in $C(Q, \mathcal{W}_i)$, which implies that $C(Q, \mathcal{W}_i)$ is CAT(0) cube complex quasiisometric to $\mathbb{R}$.
\end{proof}

We can now prove Corollary~\ref{cor:geometricFlat}.

\begin{cor} \label{cor:geometricFlat}
Let $G$ be virtually $\integers^n$.
Suppose $G$ acts properly and without inversions on a CAT(0) cube complex $X$.
Then $G$ cocompactly stabilizes a subspace  $F \subseteq X$ homeomorphic to $\mathbb{R}^n$ such that for each hyperplane $\Lambda \subseteq X$, the intersection $\Lambda \cap F$ is either empty or homeomorphic to $\mathbb{R}^{n-1}$.
\end{cor}

\begin{proof}
By Theorem~\ref{thm:main} there is a $G$-equivariant, isometrically embedded, subcomplex $Y \subseteq X$, such that $Y = \prod_{i=1}^m C_i$, where each $C_i$ is a quasiline, and $\stab_G(\Lambda)$ is a codimension-1 subgroup.
Considering $Y$ with the CAT(0) metric, note that $Y$ is a complete CAT(0) metric space in its own right, and $G$ acts semisimply on $Y$.
By the Flat Torus Theorem~\cite{BridsonHaefliger} there is an isometrically embedded flat $F\subseteq Y$.
Note that $F \subseteq X$ is not isometrically embedded.
As $\stab_G(\Lambda)$ is a codimension-1 subgroup of $G$ for each hyperplane $\Lambda$ in $X$, the intersection $\Lambda \cap F = (\Lambda \cap Y) \cap F$ is either empty or, as $F \subseteq Y$ is isometrically embedded, the hyperplane intersection is an isometrically embedded copy of $\reals^{n-1}$.
\end{proof}

\section{Haglund's Axis} \label{sec:HaglundRevisited}

The goal of this section is to reprove the following result of Haglund as a consequence of Corollary~\ref{cor:geometricFlat}.
\begin{thm}[Haglund~\cite{HaglundSemiSimple}] \label{thm:HaglundAxis}
Let $G$ be a group acting on a CAT(0) cube complex without inversions.
Every element $g \in G$ either fixes a $0$-cube of $G$, or stabilizes a combinatorial geodesic axis.
\end{thm}

\begin{proof}
As finite groups don't contain codimension-1 subgroups, Theorem~\ref{thm:fixedCube} implies that if $g$ is finite order then it fixes a $0$-cube.
Suppose that $G$ does not fix a $0$-cube, then $\langle g \rangle$ must act properly on $X$.
By Corollary~\ref{cor:geometricFlat}, there is a line $L \subset X$ stabilized by $G$, that intersects each hyperplane at most once at a single point in $L$.
Let $\mathcal{W}_L$ be the set of hyperplanes intersecting $L$.
Note that the intersection points of the walls in $\mathcal{W}_L$ with $L$ is locally finite subset.

Fix a basepoint $p \in L$ that doesn't belong to a hyperplane intersecting $L$, and let $x$ be the canonical $0$-cube corresponding to $p$.
Let $\Lambda_1, \ldots, \Lambda_k$ be the set of hyperplanes separating $p$ and $gp$, and assume that $p \in \overleftarrow{\Lambda}_i$.
Reindex the hyperplanes such that $\ola{\Lambda}_1 \cap L \subseteq \ola{\Lambda}_2 \cap L \subseteq \cdots \subseteq \ola{\Lambda}_k \cap L$.
The ordering of the hyperplanes separating $p$ and $gp$ determines a combinatorial geodesic joining $x$ and $gx$ of length $k$, where the $i$-th edge is a $1$-cube dual to $\Lambda_i$.
This can be extended $\langle g \rangle$-equivariantly, to obtain a combinatorial geodesic axis $L_c$, since each hyperplanes intersects $L_c$ at most once.
\end{proof}

\appendix

\section{Codimension-1 Subgroups} \label{AppendixA}

\begin{defn} \label{defn:codim1}
Let $G$ be a finitely generated group.
Let $\Upsilon$ denote the Cayley graph of $G$ with respect to some finite generating set.
A subgroup $H \leqslant G$ is \emph{codimension-1} if $K \slash \Upsilon$ has more than one end.

Let $\oplus$ denote the operation of symmetric difference.
A subset $A \subseteq G$ is \emph{$H$-finite} if $A \subseteq HF$ where $F$ is some finite subset of $G$.
We will use the following equivalent formulation (see~\cite{Scott77}) of codimension-1:
A subgroup $H \leq G$ is a codimension-1 subgroup if there exists some $A \subseteq G$ such that 
\begin{enumerate}
 \item \label{item:codim1:1} $A = HA$,
 \item \label{item:codim1:2} $A$ is \emph{$H$-almost invariant}, that is to say that $A \oplus Ag$ is $H$-finite for any $g \in G$.
 \item $A$ is \emph{$H$-proper}, that is to say that neither $A$ nor $G-A$ is $H$-finite.
\end{enumerate}
\end{defn}

We will reprove the following theorem from~\cite{Gerasimov97} using techniques from~\cite{NibloRoller98}.

\begin{thm}
Let $G$ be a finitely generated group acting on a CAT(0) cube complex $X$ without edge inversions or fixing a $0$-cube.
Then the stabilizer of some hyperplane in $X$ is a codimension-1 subgroup of $G$.
\end{thm}

\begin{proof}
 Suppose that no hyperplane stabilizer is a codimension-1 subgroup of $G$. We will find a $0$-cube fixed by $G$.

 Let $\mathcal{H}$ denote the set of hyperplanes in $X$.
 We can assume that $X$ has finitely many $G$-orbits of hyperplanes after possibly passing to the convex hull of a single $0$-cube orbit in $X$. 
 If $x,y$ are $0$-cubes in $X$, then let $\Delta(x,y) \subseteq \mathcal{H}$ denote the hyperplanes separating $x$ and $y$.
 Note that $$\dist_X^c(x,y) = | \Delta(x,y) |.$$
 
 Let $\Lambda_1, \ldots, \Lambda_n$ be a minimal set of representatives of those orbits.
 Let $$H_i =\stab_G(\Lambda_i) \; \textrm{ and } \; A_i = \{ g \in G \mid gx_0 \in \overleftarrow{\Lambda}_i \}. $$
 We can verify that $A_i$ satisfies the first two criteria in Definition~\ref{defn:codim1}.
 
 (\ref{item:codim1:1}): It is immediate that $A_i = H_iA_i$, as $G$ doesn't invert the hyperplanes in $X$.
 
 (\ref{item:codim1:2}): Let \emph{xor} denote the exclusive or. For $f\in G$ we can deduce that $A_i \oplus A_if$ is $H_i$-finite:
       \begin{align*}
        g \in A_i \oplus A_i f & \iff gx_0 \in \overleftarrow{\Lambda}_i \textrm{ xor } gf^{-1}x_0 \in \overleftarrow{\Lambda}_i \\
                         & \iff x_0 \in g^{-1}\overleftarrow{\Lambda}_i \textrm{ xor } f^{-1}x_0 \in g^{-1}\overleftarrow{\Lambda}_i \\
                         & \iff g\in G \textrm{ such that } g^{-1}\Lambda_i \textrm{ separates $x_0$ and $f^{-1}x_0$. }  
       \end{align*}
       As $(X, \mathcal{H})$ is a wallspace, there are only finitely many $g \in G$ such that $g^{-1}\Lambda_i$ separates $x_0$ and $f^{-1}x_0$. 
       If $g_1 \Lambda_i, \ldots, g_k\Lambda_i$ are the translates then 
       $$A_i + A_if = \{g_1, \ldots, g_k\} H_i$$
       which implies almost $H_i$-invariant.
       
      Therefore,  $A_i$ cannot be $H_i$-proper, for any $i$, as we have assumed that none of the $H_i$ are codimension-1.
    This means that $A_i$ is $H_i$-finite so 
    $A_i \subseteq H_iF_i$ where $F_i \subseteq G$ is finite.
    \begin{claim}
     $\dist_X(x_0, fx_0) <  2 \max_i(|F_i|)$ for all $f \in G$.
    \end{claim}
    \begin{proof}
     \begin{align*}
      g\Lambda_i \in \Delta(x_0, fx_0) & \iff x_0[g\Lambda_i] \neq fx_0[g\Lambda_i] \\
                                       & \iff x_0 \in g \overleftarrow{\Lambda}_i \textrm{ xor } fx_0 \in g \overleftarrow{\Lambda}_i \\
                                       & \iff g^{-1}x_0 \in g \overleftarrow{\Lambda}_i \; \textrm{ xor } \; g^{-1}fx_0 \in g \overleftarrow{\Lambda}_i \\
                                       & \iff g^{-1} \in A_i \; \textrm{ xor } \; g^{-1} \in Af^{-1}\\
                                       & \iff g^{-1} \in A_i + A_if^{-1}
     \end{align*}
     As the final set is covered by $2|F_i|$ translates of $H_i$, we can deduce that there are at most $2|F_i|$ hyperplanes in $\Delta(x_0, fx_0)$.
    \end{proof}
    
    Thus, we can conclude that the $G$-orbit of $x_0$ is a bounded set.
    If $G$ has a finite orbit in $X$, then the convex hull of the orbit is a compact, finite dimensional, complete CAT(0) cube complex, and we can apply  Corollary II.2.8 (1) from~\cite{BridsonHaefliger} to find a fixed point $p$.
    As $p$ is in the interior of some $n$-cube that is fixed by $G$, and since $G$ doesn't invert hyperplanes we can deduce that $G$ fixes a $0$-cube in that cube.
    If the $G$-orbits in $X$ are infinite, then their convex hull may not be complete, so the above argument will not hold.

Let $\mathcal{C}(\mathcal{H})$ denote the \emph{connected cube},  a graph with vertices given by functions $c: \mathcal{H} \rightarrow \{0,1\}$ with finite support, and edges that join a pair of distinct vertices if and only if they differ on precisely one hyperplane.

Fix a $0$-cube $x_0$.
Then there is an embedding $$\phi: X^1 \hookrightarrow \mathcal{C}(\mathcal{H})$$ 
that maps the $0$-cube $x$ to $c_x$ where
\[ c_x(\Lambda) = \begin{cases}
       1 & \textrm{if } \; x[\Lambda] \neq x_0[\Lambda] \\
       0 & \textrm{if } \; x[\Lambda] = x_0[\Lambda] \\
      \end{cases}
\]
\noindent A hyperplan $\Lambda \in \mathcal{H}$ \emph{separates} two vertices $c_1,c_2$ in $\mathcal{C}(\mathcal{H})$ if $c_1(\Lambda) \neq c_2(\Lambda)$.
Note that $\Lambda$ separates $0$-cubes $x,y$ in $X$ if and only if it separates $\phi(x)$ and $\phi(y)$.
Therefore, we can define $\Delta(c_1, c_2)$ for vertices in $\mathcal{C}(\mathcal{H})$ and conclude that if $x,y$ are $0$-cubes in $X$ then $\Delta(x,y) = \Delta( \phi(x), \phi(y) )$.  
This implies that $\phi$ is an isometric embedding in the combinatorial metric.
    
    We will show that a bounded orbit in $X$ implies there is a fixed $0$-cube in $\mathcal{C}(\mathcal{H})$ and then argue that we can go one step further and find a fixed $0$-cube in $X$.
 
 Let $\ell^2(\mathcal{H})$ be the Hilbert Space of square summable functions $s: \mathcal{H} \rightarrow \mathbb{R}$.
    There is an embedding $\rho: \mathcal{C}(\mathcal{H}) \rightarrow \ell^2(\mathcal{H})$ given by
    \[
     \rho(c)(\Lambda) = c[\Lambda]
    \]
    It is straight forward to verify that $\| \rho(c_1) - \rho(c_2) \|^2 = \dist_{\mathcal{C}(\mathcal{H})}(c_1, c_2)$.
    There is a $G$-action on $\ell^2(\mathcal{H})$ such that if $s \in \ell^2(\mathcal{H}), \Lambda \in \mathcal{H}, g \in G$ then 
    \[
     g s(\Lambda) = \begin{cases}
                     s(g^{-1}\Lambda) & \textrm{if $c_{x_0}(g^{-1} \Lambda) = c_{x_0}(\Lambda)$} \\
                     1 - s(g^{-1}\Lambda) & \textrm{if $c_{x_0}(g^{-1}\Lambda) \neq c_{x_0}(\Lambda)$} \\
                    \end{cases}
    \]
    It is again straight forward to verify that this action is by isometries, and that $\rho$ is $G$-equivariant.
    
    As $G{x_0}$ is bounded, so is $G(\rho \circ \phi(x_0))$.
    It then follows that $G$ has a fixed point in $\ell^2(\mathcal{H})$ (\cite{NibloRoller98} gives a proof and also cites Lemma 3.8 in~\cite{HarpeValette89}).
    Let $s: \mathcal{H} \rightarrow \mathbb{R}$ be the fixed point.
    For all $g \in G$ we can deduce that $s(g \Lambda)$ is either $s(\Lambda)$ or $1 - s(\Lambda)$. 
    Therefore $s$ can only take two values on the hyperplanes in a single $G$-orbit.
    As $s$ has to be square summable the two values have to be $0$ and $1$, and $s$ can only take the value $1$ on finitely many hyperplanes.
    Thus, $s$ is the image of a point $c$ in $\mathcal{C}(\mathcal{S})$.
    
    Let $c \in \mathcal{C}(\mathcal{S})$ be a $G$-invariant vetex which minimizes the distance to the image of $X^1$ in $\mathcal{C}(\mathcal{S})$.
    Let $Z$ be a $G$-orbit of $0$-cubes in $X$ such that $\phi(Z)$ realize the minimal distance from $c$.
    
    Let $\mathcal{V}$ be the set of hyperplanes that intersect $\{c \} \cup \mathcal{V}$.
    Every hyperplane in $ \mathcal{V}$ must intersect $Z$ otherwise if $\mathcal{F} \subseteq \mathcal{V}$ is the finite, $G$-invariant subset of hyperplanes separating $c$ from $Z$ we can define a $0$-cube $c'$ such that 
    $$c'(\Lambda) = \begin{cases}
                     c(\Lambda) & \textrm{ if $\Lambda \notin \mathcal{F}$} \\
                     1 - c(\Lambda') & \textrm{ if $\Lambda \in \mathcal{F}$} \\
                    \end{cases} $$
    and deduce that $c'$ is $G$-invariant and is $|\mathcal{F}|$ closer to $Z$ than $c$.
    
    Let $z_0, z_1, z_2, \ldots$ an enumeration of $0$-cubes in $Z$.
    Each hyperplane separating $z_0$ and $z_1$ must lie in either $\Delta(z_0, c)$ or $\Delta( z_1, c)$.
    As $z_0$ is minimal distance in $X$ from $c$, the edges in $X$ incident to $z_0$ must be dual to hyperplanes not in $\Delta(z_0, c)$, and instead belongs to $\Delta(z_1,c)$.
    Therefore, the hyperplane $\Lambda_0 \in \mathcal{V}$ dual to the first edge in a combinatorial geodesic joining $z_0$ to $z_1$ must lie in $\Delta(z_1,c)$.
    Similarly, there exists a hyperplane $\Lambda_1$ dual to the first edge of the combinatorial geodesic in $X$ joining $z_1$ to $z_2$ that belongs to $\Delta(z_2,c)$ but not $\Delta(z_1,c)$.
    Note that $\Lambda_1$ cannot intersect $\Lambda_0$ in $X$, otherwise $\Lambda_0$ would be dual to an edge incident to $z_1$, which would imply that there exists a $0$-cube in $X$ adjacent to $z_1$ that is closer to $c$.
    Therefore $\Lambda_0, \Lambda_1$ separates $z_0$ from $z_2$ in $X$.
    Iterating this argument produces a sequence of disjoint hyperplanes $\Lambda_0, \Lambda_2, \Lambda_3, \ldots$  such that $\Lambda_0, \ldots, \Lambda_k$ separates $z_0$ from $z_{k+1}$ in $X$.
    This contradicts the hypothesis that $Z$ is a bounded set in $X$.
\end{proof}

\bibliographystyle{plain}
\bibliography{Ref}

\end{document}